\documentclass[twoside,reqno,10pt]{amsart}
\usepackage{amsmath}
\usepackage{amssymb}
\usepackage{mathrsfs}
\usepackage[mathcal]{eucal}
\usepackage{color}
\usepackage[shortalphabetic, initials]{amsrefs}
\DeclareMathOperator{\lcm}{lcm}

\begin{document}

\title[Local Weyl equivalence]{On local Weyl equivalence \\ of higher order Fucshian equations}

\author[S. Tanny]{Shira Tanny}
\author[S. Yakovenko]{Sergei Yakovenko}
\address{
Department of Mathematics\\
Weizmann Institute of Science,
ISRAEL}
\email{{\{shira.giat, sergei.yakovenko\}@weizmann.ac.il}}

\begin{abstract}
We study the local classification of higher order Fuchsian linear differential equations under various refinements of the classical notion of the ``type of differential equation'' introduced by Frobenius. The main source of difficulties is the fact that there is no natural group action generating this classification.

We establish a number of results on higher order equations which are similar but not completely parallel to the known results on local (holomorphic and meromorphic) gauge equivalence of systems of first order equations.
\end{abstract}

\date{Decemeber 25, 2014}
\let\parasymbol=\S
\def\secref#1{\parasymbol\ref{#1}}
\def\C{{\mathbb C}}
\def\N{{\mathbb N}}
\edef\Ore{\O}
\def\O{{\mathscr O}}
\def\M{{\mathscr M}}
\def\W{{\mathscr W}}
\def\G{{\mathscr G}}
\def\R{{\mathscr R}}
\def\F{{\mathscr F}}
\def\Z{{\mathbb Z}}
\def\P{{\mathbb P}}
\def\E{{\scriptstyle{\mathscr E}}}
\def\H{\boldsymbol H}
\def\Ker{\operatorname{Ker}}
\def\:{\colon}
\def\eu{{\epsilon}}
\def\f{\varphi}
\def\l{\lambda}
\def\L{\varLambda}
\def\ord{\operatorname{ord}}
\def\Mat{\operatorname{Mat}}
\def\^{\hat}
\def\~{\widetilde}
\def\ssm{\smallsetminus}
\def\le{\leqslant}
\def\ge{\geqslant}
\def\GL{{\operatorname{GL}}}
\def\id{\operatorname{id}}
\edef\iorig{\i}
\def\i{\mathrm i}
\def\d{\partial}
\def\<{\left<}
\def\>{\right>}
\def\S{\varSigma}
\def\Arg{\operatorname{Arg}}
\def\lcm{\operatorname{lcm}}
\def\sh#1{^{\scriptscriptstyle[#1]}}
\def\supp{\operatorname{supp}}

\newtheorem{Thm}{Theorem}
\newtheorem{Lem}{Lemma}
\newtheorem{Prop}{Proposition}
\newtheorem{Cor}{Corollary}

\theoremstyle{definition}
\newtheorem{Def}{Definition}
\newtheorem{Prob}{\textcolor{red}{Problem}}
\newtheorem{Ex}{Example}

\theoremstyle{remark}
\newtheorem{Rem}{Remark}

\maketitle

\section{Local classification of linear ordinary differential equations}

\subsection{Systems and higher order equations}
The local analytic theory of linear ordinary differential equations exists in two parallel flavours, either that of systems of several first order equations, or of scalar (higher order) equations. One can relatively easily transform one type of objects to the other, yet this transformation loses some additional structures.

Let $\Bbbk$ be a differential field, called the \emph{field of coefficients}. We will be interested almost exclusively in the field $\M=\M(\C^1,0)$ of \emph{meromorphic germs} at the origin $t=0$ on the complex line $\C=\C^1$, the quotient field of the ring $\O=\O(\C^1,0)$ of holomorphic germs at the origin. The standard $\C$-linear derivation $\d=\frac{d}{dt}$ acts on both $\O$ and $\M$ according to the Leibniz rule and extends on vector and matrix functions with entries in $\Bbbk$ un the natural way.

Let $A\in\Mat(n,\Bbbk)$ be an $(n\times n)$-matrix function, called the \emph{coefficients matrix}, $A=\|a_{ij}(t)\|_{i,j=1}^n$, $a_{ij}\in\Bbbk$. This matrix defines the homogeneous system of  linear ordinary equations
\begin{equation}\label{ls}
\d x=Ax,\qquad x=(x_1,\dots,x_n)\in\C^n, \quad t\in(\C,0).
\end{equation}
The system \eqref{ls} only exceptionally rarely has a solution $x\in\Bbbk^n$. However, it always has $n$ linear independent solutions in the class of functions analytic in a small punctured neighborhood of the origin, which are multivalued (ramified) over the point $t=0$. Assembling these solutions (as column vectors) into a multivalued matrix function $X=X(t)$ whose determinant never vanishes for $t\ne0$, we can without loss of generality reduce the system \eqref{ls} to one \emph{matrix differential} equation $\d X=AX$. For instance, the \emph{trivial} system is defined by the equation $\d X=0$, and any invertible constant matrix $C\in\GL(n,\C)$ is its solution.

Alternatively, one may consider homogeneous linear ordinary differential equations of the form
\begin{equation}\label{le}
a_0\d^n u+a_1\d^{n-1}u+\cdots+a_{n-1}\d u+a_n u=0,\qquad a_0,\dots,a_m\in\Bbbk, \ a_0\ne0.
\end{equation}
Each equation \eqref{le} is a linear (over $\Bbbk$) relation between the unknown function $u$ and its derivatives $\d^k u$ up to order $k=n$. Traditionally, such equations are written using \emph{linear differential operators}: if $L=\sum_0^n a_i\d^{n-i}\in\Bbbk[\d]$ is the formal expression, then the above equation is written under the form $Lu=0$. Elements of the field $\Bbbk$ are identified with ``operators of zeroth order'' $u\mapsto au$, $a\in\Bbbk$. The key feature of differential operators is the possibility of their composition which equips the $\Bbbk$-space of linear operators with the structure of (noncommutative infinite-dimensional) $\C$-algebra, denoted by $\W$.\footnote{The classical Weyl algebra is generated over $\C$ by two elements $t,\d$ with the commutativity relation $[\d,t]=1$. It embeds naturally into the algebra $\Bbbk[\d]$ for the differential field of rational functions $\Bbbk=\C(t)$.}

As before, generically solution exists only as a multivalued function defined for $t\ne 0$ and ramified over the origin.

\subsection{Mutual reduction}\label{sec:mutual}
One can easily transform the equation \eqref{le} to a system \eqref{ls} by introducing the variables $x_k=\d^{k-1}u$, $k=1,\dots,n$. The corresponding first order identities take the form
\begin{equation}\label{companion}
\d x_k=x_{k+1},\quad k=1,\dots,n-1,\qquad \d x_{n}=-a_0^{-1}(a_1x_{n-1}+\cdots+a_nx_1).
\end{equation}
Conversely, each of the variables $u=x_k$ of a solution $x(t)$ to the system \eqref{ls} satisfies an equation of the form \eqref{le}. To obtain this equation, note that all derivatives $\d^i u$ are $\Bbbk$-linear combinations of the formal variables $x_1,\dots,x_n$. Indeed, by induction, if $\d ^i x=A_i x$, $A_i\in\Mat(n,\Bbbk)$, $A_0=E$, $A_1=A$, then
\begin{equation}
\d^{i+1} x=(\d A_i)x+A_iA_1x=(\d A_i+A_iA)x=A_{i+1}x,\qquad i=1,2,\dots.
\end{equation}
Taking the $k$th line of these identities yields the required liner combination. Since the space of combinations is $n$-dimensional (over $\Bbbk$), we conclude  that $n+1$ derivatives $u,\d u,\d^2 u, \dots,\d^n u$ are necessarily linear dependent over $\Bbbk$ (the order can be less than $n$). This dependence is of the form \eqref{le}, but the corresponding equation will in general depend on the choice of $k$ between $1$ and $n$. Slightly modifying this construction, one can produce a differential equation of order $\le n^2$, satisfied by \emph{all components} $x_{ij}$ of any fundamental matrix solution $X=\|x_{ij}\|$ of the equation $\d X=AX$.

\subsection{Gauge equivalence of linear systems. Equations of the same type}\label{sec:gauge}
The group $\G=\GL(n,\Bbbk)$ of invertible matrix functions with entries in the field $\Bbbk$ acts naturally on the space of all linear systems of the form \eqref{ls}. Namely, if $H=\|h_{ij}(t)\|_{i,j=1}^n$, $h_{ij}\in\Bbbk$, is such a function with the inverse $H^{-1}\in\GL(n,\Bbbk)$, then one can ``change variables'' in \eqref{ls} by substituting $y=Hx$, $y=(y_1,\dots,y_n)\in\C^n$. This substitution transforms \eqref{ls} to the identity $\d y=(\d H)x+H\d x=(\d H)H^{-1}y+HAH^{-1}y$, so that
\begin{equation}\label{gauge}
\d y=By,\qquad B\in\Mat(n,\Bbbk),\quad (\d H)\cdot H^{-1}+HAH^{-1}.
\end{equation}
This differs from the conjugacy of linear operators by the \emph{logarithmic derivative} $(\d H)\cdot H^{-1}$; this term vanishes if $H$ is constant.

Two systems $\d x=Ax$ and $\d y=By$ are called \emph{gauge equivalent}, if there exists an element $H\in\G$ such that \eqref{gauge} holds. Since $\G$ is a group, this equivalence naturally is reflexive, symmetric and transitive. Thus one can formulate the problem of \emph{classification}: what is the simplest form to which a given linear system can be transformed by a suitable gauge transformation? The corresponding theory is fairly well established, see below for the initial results.

\begin{Rem}
Systems of linear equations \eqref{ls} can be considered geometrically as flat meromorphic connections on a vector bundle over the (complex) 1-dimensional base. The gauge transform corresponds to the change of a tuple of horizontal sections locally trivializing this bundle. Such interpretation allows for global and multidimensional generalizations, see \cite{thebook}*{Chapter III} and \cite{mero-flat}.
\end{Rem}

Unfortunately, the notion of gauge equivalence is too restricted to deal with high order equations: indeed, since the unknown function is scalar, only the transformations of the form $u=hv$, $h\in\Bbbk$, can be considered, but one cannot expect this small group to produce a meaningful classification.

Instead it is natural to consider $\Bbbk$-linear changes of variables of a more general form which involve the unknown function and its derivatives. More specifically, one can choose a tuple of functions $h=(h_0,\dots,h_{n-1})\in\Bbbk^n$ and use it to change the dependent variable from $u$ to $v$ as follows,
\begin{equation}\label{optrans}
v=h_1\d^{n-1}u+h_2\d^{n-2}u+\cdots+h_{n-1}\d u+h_{n}u.
\end{equation}
The reason why derivatives of order $n$ and may be omitted, is rather clear: if the transformation \eqref{optrans} is applied to an equation \eqref{le} of order $n$, then all such higher order derivatives can be replaced by $\Bbbk$-linear combinations of the lower order derivatives by virtue of the equation.

The new variable $v$ also satisfies a linear differential equation which can be derived as follows (cf.~with \secref{sec:mutual}). Differentiating the formula \eqref{optrans} for $v$ by virtue of the equation \eqref{le}, one can see that all higher order derivatives $\d^i v$ can be expressed as linear combinations (over $\Bbbk$) of the formal derivatives $\d^j u$, $u=0,\dots,n-1$. The space of such combinations is $n$-dimensional, so no later than on the $n$th step there will necessary appear an identity of the form $b_0\d^m v+b_1\d^{m-1}v+\cdots+b_{m-1}\d v+b_mv=0$, $b_0\ne0$, $b_j\in\Bbbk$, $m\le n$, which is the transform of the equation \eqref{le} by the action of \eqref{optrans}. Classically, the initial equation and the transformed equation are called \emph{equations of the same type}, see \cite{singer,tsarev,ore}, but we would prefer to use the term ``Weyl equivalence'' (justifying the fact), with an intention to refine it by imposing additional restrictions imposed on the transformation \eqref{optrans}.

In order for this change of variables to be ``faithful'', one has to impose the additional condition of nondegeneracy: \emph{no solution of \eqref{le} is mapped into identical zero by the transformation \eqref{optrans}}. Indeed, if this extra assumption is violated, one can easily transform the initial equation to the trivial (meaningless) form $0=0$. On the other hand, accepting this condition guarantees (as can be easily shown) that the transformed equation has the same order $m=n$.

Still a few questions remain unanswered by this na{\"\iorig}ve approach. The transformation \eqref{optrans}, unlike the gauge transformation of linear systems, is rather problematic to invert: transition from $u$ to $v$ always has a nontrivial kernel (solutions of the corresponding homogeneous equations). In addition, ``restoring'' $u$ from $v$ is in general a transcendental operation requiring integration of linear equations, and it is by no means clear how one should proceed.

The algebraic nature of these questions was studies since 1880-ies by F. Frobenius, E. Landau, A. Loewy, W. Krull and culminated in the perfect form in the brilliant paper by \Ore.~Ore \cite{ore}. The idea is to consider the noncommutative algebra of differential operators $\Bbbk[\d]$ with coefficients in $\Bbbk$. The next section \secref{sec:ore} summarizes the necessary fundamentals of the ``algebraic theory of noncommutative polynomials'' following \cite{ore}.

\subsection{Singularities, monodromy}\label{sec:analysis}
From this moment we focus on the special case where $\Bbbk=\M$ is the differential field of meromorphic germs at the origin and denote for brevity $\W=\M[\d]$ the algebra of operators with meromorphic coefficients.

For each linear system \eqref{ls} or a high order equation \eqref{le} with meromorphic coefficients one can choose representatives of germs of all coefficients $a_{ij}(t)$, resp., $a_i(t)$ in a punctured neighborhood of the origin $(\C^1,0)\ssm\{0\}$ so small that all representatives are holomorphic in this punctured neighborhood. The classical theorems of analysis guarantee that solutions of the system (resp., equation) are holomorphic on the universal cover of this punctured neighborhood, i.e., in the more traditional terminology, are multivalued analytic functions on $(\C^1,0)$ ramified at the origin.

If the coefficients of the system \eqref{ls} are holomorphic at the origin, i.e., $A\in\Mat(n,\O)\subsetneq\Mat(n,\M)$, then for the same reasons solutions of the system are holomorphic (hence single-valued) at the origin. This case is called \emph{nonsingular}, and the corresponding matrix equation admits a unique solution $X\in\GL(n,\O)$ with the initial condition $X(0)=E$ (the identity matrix).

Solution $X$ of a general matrix equation $\d X=AX$ with $A\in\GL(\M,n)$ after continuation along a small closed loop around the origin gets transformed into another solution $X'=XM$ of the same equation. The \emph{monodromy matrix} $M\in\GL(n,\C)$ depends on $X$.

A homogeneous equation \eqref{le} defined by a linear operator $L=\sum_{i=0}^n a_i\d^{n-i}$ can always be multiplied by a meromorphic multiplier so that all its coefficients become holomorphic and at least one of them is nonvanishing at the origin. The reduction \eqref{companion} shows that if it is the leading coefficient $a_0$  that is nonvanishing, then all solutions of the equation $Lu=0$ are holomorphic at the origin (we call such operators \emph{nonsingular}), otherwise they may be ramified at the origin.

Choose a neighborhood $U=(\C^1,0)$ and meromorphic representatives of the germs $a_j(\cdot)$ which have no other poles in $U$ expect for $t=0$. If $0\ne t_0\in U$ is any other point in the domain of the system (equation), then it is well known that germs of solutions of the system (equation) $Lu=0$ form a $\C$-linear subspace in $Z_L\subset\O(\C,t_0)$ of dimension $\dim_\C Z_L$ exactly equal to $n$. After the analytic continuation along a small loop around the origin, this space is mapped into itself by a linear invertible map called the \emph{monodromy transformation} (monodromy, for short): for any basis $u_1,\dots,u_n$ in the space of solutions (considered as a row vector function), we have
\begin{equation}\label{monodromy}
\Delta \begin{pmatrix}u_1&\cdots& u_n\end{pmatrix}=\begin{pmatrix}u_1&\cdots& u_n\end{pmatrix}M
\end{equation}
for a suitable nondegenerate matrix $M$ (depending on the basis $\{u_i\}_{i=1}^n$).

\subsection{Different flavors of the gauge classification}\label{sec:sys}
The gauge transformation group $\G=\GL(n,\M)$ introduced above, may be too large for certain problems of analysis, see \secref{sec:analysis}. For several reasons it is interesting to consider a smaller group $\G_h=\GL(n,\O)$ of \emph{holomorphic} matrix functions which are holomorphically invertible. It is the semidirect product of $\GL(n,\C)$ and the group $\G_0$ of holomorphic matrix germs $H$ which are identical at the origin, $\G_0=\{H\in\G:H(0)=E\}$.

Besides, one can identify two types of singularities of linear systems, characterized by strikingly different behavior of  solutions, called respectively  \emph{regular} (in full, regular singular, to avoid confusion with nonsingular systems) and \emph{irregular} singularities. Recall \cite{thebook}*{Definition 16.1} that the system \eqref{ls} is called regular if the norm $|X(t)|$ of any its fundamental matrix solution grows no faster than polynomially when approaching the singular point in any sector on the $t$-plane (more precisely, on the universal cover of $(\C^1,0)\ssm 0$):
\begin{equation}\label{moderate}
|X(t)|\le Ct^{-N} \quad\forall t\in(\C^1,0),\ \alpha<\Arg t<\beta,\qquad C>0,\ N<+\infty,
\end{equation}
for some constants $C,N$ depending on the sector (its opening and the radius). This condition is difficult to verify as it refers to the properties of solutions, but it is automatically satisfied for \emph{Fuchsian} systems, when the meromorphic matrix function $A$ has a pole of at most first order \cite{thebook}*{Theorem 16.10 (Sauvage, 1886)}.

\begin{Ex}
An \emph{Euler system} is any system of the form $\d X=t^{-1}BX$ with a \emph{constant} matrix $B\in\Mat(n,\C)$. Its fundamental matrix solution is given by the (multivalued) matrix function $X(t)=t^B=\exp(B\ln t)$, $t\in(\C,0)$. If $B=\operatorname{diag} (\l_1,\dots,\l_n)$ is a diagonal matrix with $\l_i\in\C$, then the solution is also diagonal, $X(t)=\operatorname{diag}(t^{\l_1},\dots,t^{\l_n})$. The monodromy matrix of this solution is $\exp 2\pi i B\in\GL(n,\C)$.

In general if $\l_1,\dots,\l_n$ are the eigenvalues of the matrix $B$, then the corresponding Euler system is called \emph{resonant} if some of the differences $\l_i-\l_k$ are natural numbers (nonzero), otherwise the system is called \emph{nonresonant}.
\end{Ex}

The principal results on classification of linear systems are summarized in Table~1, based on \cite{thebook}*{\parasymbol 16, \parasymbol 20}.

\begin{table}
\caption{Normal forms of linear systems.}
\renewcommand{\arraystretch}{1.5}
\begin{tabular}{|c | c | c|}
\hline
Type of singularity / Group & Holomorphic $\G_0$ & Meromorphic $\G$ \\ \hline
Nonsingular & Trivial & Trivial \\ \hline
Fuchsian nonresonant & Euler   & \\
Fuchsian resonant & Polynomial integrable  & Euler \\ \cline{1-2}
Regular non-Fuchsian & Rational  & \\ \hline
Irregular nonresonant & \multicolumn{2}{|c|}{Formally diagonalizable, divergent  } \\ \cline{2-3}
Irregular resonant & \multicolumn{2}{|c|}{Ramified gauge transforms are required } \\ \hline
\end{tabular}
\medskip
\parbox{\textwidth}{\begin{small}\advance\baselineskip by -2pt
\subsubsection*{Polynomial normal form} The system takes the form  $\d X=t^{-1}(B_0+B_1t+B_2t^2+\cdots+B_pt^p)X$, where $p$ is the maximal \emph{integer} difference between the eigenvalues of the Jordan matrix $B_0$. The matrices $B_k$ may have nonzero entry in the $(i,j)$th position only if $\l_i-\l_j=k$, that is, are very sparse. The system in the normal form can be explicitly solved: there exists a fundamental matrix solution of the form $X(t)=t^Pt^Q$ with two constant matrices $P,Q\in\Mat(n,\C)$ \emph{not commuting} between themselves.
\subsubsection*{Rational normal norm} In this case the normal form is rational and explicit but its description is off the main track of this work.
\subsubsection*{Irregular systems}
For irregular systems with the matrix of coefficients represented by a Laurent series $A(t)=t^{-r}(A_0+tA_1+\cdots)$, $r\ge 2$, the definition of resonance requires that the eigenvalues of the leading matrix coefficient $A_0$ are pairwise different. In the nonresonant case one can find a formal matrix series $H(t)=E+H_1t+H_2t^2+\cdots$ which reduces the system to a diagonal normal form $\d X=t^{-r}D(t)X$ with a diagonal polynomial normal form, $D(0)=A_0$, but this series almost always diverges, see \cite{thebook}*{\parasymbol 20}. To deal with the resonant case, one has to consider gauge transformations with entries being themselves ramified, i.e., involving noninteger powers of $t$. We will not deal with irregular systems or equations in this paper.\par
\end{small}}
\end{table}

The notions of (ir)regularity can be defined also for linear equations of higher order. Somewhat mysteriously, unlike in the case of general linear systems, it is \emph{equivalent} to a condition on the order of the poles of the ratios $a_i/a_0\in\M$ of the coefficients of the equation (this condition is also called the Fuchsian condition).

\subsection{Goals of the paper and main results}
We study the classification of nonsingular or Fuchsian (singular) equations with respect to the Weyl equivalence (formally introduced below).

It can be easily shown (see below) that nonsingular equations are Weyl equivalent to the trivial equation $\d^n u=0$, whose solutions are polynomials of degrees $\le n-1$. An equally simple fact is the Weyl equivalence of any Fuchsian equation to an Euler equation. Furthermore, we show that the property of a Fuchsian equation to possess only holomorphic (or meromorphic) solutions can be expressed in terms of Weyl equivalence.

In our paper we introduce a more fine  \emph{Fuchsian equivalence}, or $\F$-equivalence for short, using expansion of operators in noncommutative Taylor series. It turns out that the corresponding classification of Fuchsian operators is very similar to the holomorphic classification of Fuchsian systems. In particular, in the nonresonant case any Fuchsian equation is $\F$-equivalent to an Euler equation, while resonant operators are $\F$-equivalent to operators with polynomial coefficients, i.e., from $\C[t][\d]$. Finally, we show that any (resonant) Fuchsian operator is $\F$-equivalent to an operator which is \emph{Liouville integrable}, that is, whose solutions can be obtained from rational functions by iterated integration and exponentiation.

\subsection{Acknowledgements}
We are grateful to all our friends and colleagues who helped us to identify classical sources and thus avoid re-inventing the bicycle. The primary gratitude goes to Yuri Berest, Gal Binyamini, Dima Gourevitch, Dmitry Novikov, Michael Singer and Yuri Zarkhin.

\section{Algebras of differential operators}

In this section we recall the basic facts about the algebra of differential operators with coefficients from a different field.

\subsection{Noncommutative polynomials in one variable over a differential field}\label{sec:ore}
Consider the $\C$-algebra $\Bbbk[\d]$ generated by the differential field $\Bbbk$ and the symbol $\d$ with the noncommutative multiplication satisfying the Leibniz rule,
\begin{equation}\label{leibniz}
\d\cdot a=a\cdot \d+a',\qquad a,a'\in\Bbbk,\quad a'=\d a=\text{ the derivative of }a.
\end{equation}
This algebra can be considered as the algebra of differential operators acting on ``test functions'', where elements from $\Bbbk$ act by multiplication $u\mapsto au$ and $\d$ is the derivation. The operation corresponds to the composition of operators (and the dot will be omitted from the notation).

Any operator from $\Bbbk[\d]$ can be uniquely represented under the ``standard form''
\begin{equation}\label{ldo}
 L=a_0\d^n+a_1\d^{n-1}+\cdots+a_{n-1}\d+a_n,\qquad a_0,\dots,a_n\in\Bbbk,\ a_0\ne 0
\end{equation}
with the coefficients $a_i$ to the left from the powers of $\d$. The number $n\ge 0$ is called the order of the operator $L$. The composition $LM$ of two operators $L$ and $M=b_0\d^m+\cdots\in\Bbbk[\d]$ of orders $n$ and $m$ is an operator of order $n+m$ with the (nonzero) leading coefficient $a_0b_0\in\Bbbk$.

The key property of the algebra $\Bbbk[\d]$ is the possibility of division with remainder. Indeed, if $n=\ord L\ge m=\ord M$, then the difference $L-a_0b_0^{-1}\d^{n-m}M$ is an operator with zero (absent) ``leading coefficient'' before $\d^{n}$, i.e., is of order strictly less than $n$. Iterating this order depression, one can find two operators $Q,R\in\Bbbk[\d]$ such that
\begin{equation}\label{divrem}
L=QM+R,\qquad \ord Q=\ord L-\ord M,\quad \ord R<\ord M.
\end{equation}
When $R=0$ we say that $L$ is \emph{divisible} by $M$.

This construction allows to define for any two operators $L,M\in\Bbbk[\d]$ their \emph{greatest common divisor} $D=\gcd(L,M)$ as the operator of maximal order which divides both $L$ and $M$ (this operator is defined modulo a multiplication by an element from $\Bbbk$). The Euclid algorithm \cite{ore}*{Theorem 4} guarantees that for any $L,M$ there exist $U,V\in\Bbbk[\d]$ such that
\begin{equation}\label{gcd}
UL+VM=\gcd(L,M),\qquad \ord U<\ord M,\ \ord V<\ord L.
\end{equation}
A less direct computation allows to construct the \emph{least common multiple} $\lcm(L,M)$ which is by definition the smallest order operator divisible by both $L$ and $M$ (and also defined modulo a nonzero coefficient from $\Bbbk$). Indeed, consider the operators $M,\d M,\d^2M,\dots,\d^{n}M$ modulo $L$, i.e., their remainders after division by $L$, $n=\ord L$. Since all these $n+1$ remainders are of order $\le n-1$, they must be linear dependent over $\Bbbk$, that is, a certain linear combination $(c_0\d^n +\cdots+c_{n-1}\d+ c_{n})M=PM$ must be divisible by $L$: $PM=QL$, $\ord P\le \ord L$, $\ord Q\le \ord M$. There is an explicit formula expressing $\lcm(L,M)$ through the operators appearing in the Euclid's algorithm, see  \cite{ore}*{Theorem 8}.

\subsection{Algebra vs.~analysis}
Denote by $\W$ the local Weyl algebra $\Bbbk[\d]$ in the case where $\Bbbk=\M$ is the differential field of meromorphic germs.

If an operator $L$ is divisible by $M$ in $\W$, then their spaces of solutions $Z_L$, resp., $Z_M$, are subject to the inclusion $Z_M\subseteq Z_L$. Conversely, if for two operators $L,M\in\W$ we have $Z_M\subseteq Z_L$, then $L$ is divisible by $M$. Indeed, otherwise the remainder of division of $L$ by $M$ would be an operator of order strictly less than $\ord M$, whose solutions form the space of superior dimension $\dim Z_M=\ord M$.  In terms of solutions,
\begin{equation}\label{gcd-lcm}
\begin{aligned}
D=\gcd(L,M)&\iff Z_D=Z_L\cap Z_M,
\\
P=\lcm(L,M)&\iff Z_P=Z_L+Z_M
\end{aligned}
\end{equation}
(the sum of linear subspaces in $\O(\C,t_0)$ is assumed).

Thus two equations $Lu=0$ and $Mv=0$ are of the same type in the sense of \secref{sec:gauge}, if their order is the same and there exists an operator $H\in W$ which maps $Z_L$ to $Z_M$ isomorphically: for any $u$ such that $Lu=0$, the function $v=Hu$ is annulled by $M$.

\begin{Def}\label{def:we}
Two operators $L,M\in\W$ of the same order $n$ are called \emph{Weyl equivalent} (or Weyl conjugate), if there exist two operators $H,K\in\W$ of order $\le n-1$, such that
\begin{equation}\label{st}
MH=KL,\qquad \gcd(L,H)=1,\qquad \ord H,K<\ord L,M.
\end{equation}
The operator $H$ is said to be the \emph{conjugacy} between $L$ and $M$.
\end{Def}

\begin{Rem}
\Ore.~Ore uses the notation $M=\lcm(L,H)H^{-1}=HLH^{-1}$ to denote the fact of conjugacy to stress its resemblance with the ``similarity'' in the noncommutative algebra $\W$. It has its mnemonic advantages, although the formal construction of (noncommutative) field of ratios for $\W$ requires additional efforts \cite{ore}*{p.~487}.
\end{Rem}

We will abbreviate the words ``Weyl equivalence'' (resp., conjugacy) to $\W$-equivalence (conjugacy) for simplicity.

\begin{Thm}\label{thm:we}
$\W$-conjugacy is indeed an equivalence relation\textup: it is reflexive, transitive and symmetric.
\end{Thm}

\begin{proof}
It is obvious that this relationship is reflexive (suffices to choose $H=K=1$). To prove its transitivity, assume that $L_1$ is $\W$-conjugate with $L_2$, and $L_2$ with $L_3$. This means that there exist operators $H_i,K_i\in\W$, $i=1,2$, of order $\le n-1$ such that $L_2H_1=K_1L_1$ and $L_3H_2=K_2L_2$. Then $L_3H_2H_1=K_2K_1L_1$. To produce a pair $(H',K')$ conjugating $L_1$ with $L_3$, it suffices to define $H'=H_2H_1\bmod L_1$: the order of this remainder will not exceed $n-1$ by construction. One has to check that $\gcd(H',L_1)=1$, but this is obvious: if $u$ is a nontrivial solution of $L_1u=0$ and $\gcd(H_1,L_1)=1$, then $v=H_1u$ is a nontrivial solution of $L_2v=0$, hence $H_2v\ne0$. Replacing $H_2H_1$ by its remainder modulo $L_1$ cannot change the fact that $H_2H_1u\ne0$ for any solution of $L_1u=0$.

The symmetry is less trivial, see \cite{ore}*{Theorem 13}. For the reader's convenience we provide here a short direct proof due to Yu.~Berest. It is convenient to formulate it as a separate lemma.
\end{proof}

\begin{small}
\begin{Lem}\label{lem:berest}
For any two operators $L,M$ satisfying \eqref{st}, there exists a pair of operators $V,W\in\W$ such that $LV=WM$ and $\gcd(V,M)=1$.
\end{Lem}

\begin{proof}
By \eqref{gcd}, the condition $\gcd(L,H)=1$ implies that there exist $U,V\in \W$ such $UL+VH=1$. Multiplying this identity by $L$ from the left, we see that $(LU-1)L=LVH$, that is, the operator $Q$, expressed by each side of the identity, is divisible by both $H$ and $L$. This means that the operator $LVH$ is divisible by $P=\lcm(H,L)$, which in turn has two representations, $P=MH=KL$ as in \eqref{st}. The last divisibility means that $LVH=WP=WMH$ in $\W$. Yet since $\W$ has no zero divisors (the leading coefficient of any composition is nonzero), we can cancel $H$ and arrive at the identity $LV=WM$. It is a simple exercise to see that $\gcd(V,M)=1$.\par
\end{proof}
\end{small}

\subsection{Nonsingular operators}
An operator $L\in\W$ of the form \eqref{ldo} is referred to as \emph{nonsingular}, if all its coefficients are holomorphic, $a_i\in\O(\C,0)$, and the leading coefficient is invertible, $a_0(0)\ne0$. Nonsingular operators can be reduced by the transformation \eqref{companion} to a holomorphic (nonsingular) system of first order equations. An immediate conclusion is that the corresponding equation $Lu=0$ has only holomorphic solutions, and a fundamental system of solutions $\{u_k\}_{k=0}^n$ can always be chosen so that $u_k(t)=t^k+\cdots$ where the dots stand for terms of order greater than $k$.

\subsection{Fuchsian operators}
There exists another special subclass of linear operators $L\in\W$ with the property that the respective linear equations $Lu=0$ enjoy a certain regularity, namely, all their solutions grow moderately when approaching the singular point at the origin. Unlike the general linear systems \eqref{ls}, such operators admit precise algebraic description. It can be given in several equivalent forms.

Note that together with the ``basic'' derivation $\d$ any other element $a\d\in\W$ is also a derivation of the field $\M$ ($\C$-linear self-map satisfying the Leibniz rule). It can be used as the generator of the algebra $\W$. We will be mostly interested in the \emph{Euler derivation} $\eu=t\d\in\W$ with the commutation rule
\begin{equation}\label{euler-der}
\eu=t\cdot\d, \qquad \eu \cdot t^m=t^m\cdot (\eu+m),\quad\forall m\in\Z,
\end{equation}
cf.~with \eqref{leibniz}. Though the derivations $\d$ and $\eu$ are very simply related, their algebraic nature is radically different. Restricted on the finite-dimensional subspace of polynomials of any finite degree, the standard derivation $\d$ is nilpotent, while the Euler derivation is semisimple ($\eu$ is diagonal in the monomial basis, $\eu (t^m)=mt^m$).

For any polynomial $w\in\C[\eu]$ in the variable $\eu$ denote by $w\sh j$, $j\in\Z$, the shift of the argument:
\begin{equation}\label{shift}
 w\mapsto w\sh j,\quad w\sh j(\eu)=w(\eu+j),\qquad j\in\Z.
\end{equation}
This operator preserves the degree of the polynomial, and using it one can rewrite the commutation rule \eqref{euler-der} as follows,
\begin{equation}\label{shift-1}
 \forall w\in\C[\eu],\ \forall j\in\Z,\qquad wt^j=t^jw\sh j.
\end{equation}

Substituting $\d=t^{-1}\eu$ and re-expanding terms, any operator $L\in\W$ can be represented under the form
\begin{equation}\label{ff}
L=r_0\eu^n+r_1\eu^{n-1}+\cdots+r_{n-1}\eu+r_n,\qquad r_i\in\M,\ r_0\ne 0.
\end{equation}

\begin{Def}
An operator $L$ is called \emph{Fuchsian}, if in the representation \eqref{ff} all coefficients $r_i$ are holomorphic and the leading coefficient $r_0$ is invertible (nonvanishing):
\begin{equation}\label{fuchs}
r_0,r_1,\dots, r_n\in\O(\C^1,0),\qquad r_0(0)\ne 0.
\end{equation}
The operator is pre-Fuchsian, if it has a form $hL$ with any nonzero $h\in\M$; without loss of generality one may assume that $h=t^k$, $k\in\Z$.

An operator is called \emph{Eulerian}, if all coefficients $r_0,\dots,r_n\in\C$ are constant.
\end{Def}

\begin{Rem}
In the classical literature the notion of Fuchsian operators is not defined, only the notion of a (homogeneous) Fuchsian equation of the form $Lu=0$ is discussed. Clearly, two operators $L\in\W$ and $hL$, $h\in\M$, define the same homogeneous equation. For operators written in the form \eqref{ff}, the corresponding homogeneous equation will be Fuchsian if and only if all ratios $r_i/r_0\in\M$ are actually holomorphic at $t=0$ for all $i=1,\dots,n$. Our choice may seem to be artificial, yet it is justified by subsequent computations.
\end{Rem}

We will denote by $\F\subset\W$ the set of all Fuchsian operators. It is convenient to assume that holomorphically invertible germs and meromorphic germs belong in $\F$, resp., pre-$\F$ as ``differential operators of zero order''.

A Fuchsian differential equation $Lu=0$ with $L$ as in \eqref{ff} can be reduced to a Fuchsian system in the sense \eqref{sec:sys} by slightly modifying the computation \eqref{companion}: one has to introduce the new variables as follows, $x_1=u$, and then
\begin{equation}\label{f-companion}
\eu x_k=x_{k+1},\ k=1,\dots, n-1,\quad \eu x_{n}=-r_0^{-1}(r_1 x_{n-1}+\cdots+r_n x_1)
\end{equation}
(recall that $r_0$ is invertible hence $r_0^{-1}\in\O$), or in the matrix form, $\eu x=Rx$, with the holomorphic matrix $R\in\Mat(n,\O)$ of coefficients. This computation explains the relation of two Fuchsian objects of different nature. However, unlike the case of systems, in the case of scalar equations the Fuchsian condition is not only sufficient, but also necessary for the regularity (moderate growth of solutions).

\begin{Thm}[L.~Fuchs (1868), see \cite{thebook}*{Theorem 19.20}]
The operator $L\in\W$ is pre-Fuchsian if and only if all solutions of the equation $Lu=0$ and all their derivatives grow at most polynomially in any sector with the vertex at the origin in the sense \eqref{moderate}.\qed
\end{Thm}

\subsection{First results on $\W$-classification}
The initial results on $\W$-equivalen\-ce are completely parallel to $\G_h$-classification of nonsingular systems and $\G$-classification of regular systems: even the ideas of the proofs remain the same.

\begin{Thm}
A nonsingular operator is $\W$-conjugate to the operator $M=\d^n$ by a \emph{nonsingular} operator $H$ of order $n-1$.
\end{Thm}

\begin{proof}
Any nonsingular equation $Lu=0$ of order $n$ always admits $n$ linear independent solutions of the form $u_k(t)=t^{k-1}(1+\cdots)$, $k=1,\dots, n$. Indeed, one should look for solutions of the companion system \eqref{companion}  with a suitable initial condition $x_k(0)=1$, $x_j(0)=0$ for all $j\ne k$.

A linear operator $H$ transforming solutions $v_k=t^{k-1}$ of the equation $\d^n=0$ to solutions of the equation $Lu=0$ by the formulas \eqref{optrans} can be obtained by the method of indeterminate coefficients: $H=h_1\d^{n-1}+\cdots+h_{n-1}\d+h_n$. The equations $Hv_k=u_k$, $k=1,\dots, n$ correspond to a system of linear algebraic equations over $\O$ for the unknown coefficients $h_i$:
\begin{equation*}
\begin{pmatrix}
h_{n}&h_{n-1}&\cdots&h_1
\end{pmatrix}
\begin{pmatrix}
1&t&t^2&\cdots&t^{n-1}
\\
& 1& 2t&\cdots&(n-1)t^{n-2}
\\
& & 2 & &\vdots
\\
&&&\ddots&\vdots
\\
&&&&(n-1)!
\end{pmatrix}
=
\begin{pmatrix}
v_1&v_2&\cdots&v_n
\end{pmatrix}
\end{equation*}
The matrix $J$ of coefficients, the companion matrix of the tuple of solutions $v_1=1$, $v_2=t,\dots,v_n=t^{n-1}$, is holomorphic and invertible (it is upper triangular with nonzero diagonal entries). A simple inspection shows that the leading coefficient $h_1$ cannot vanish at $t=0$, hence the operator $H$ will be nonsingular.
\end{proof}

A minor modification of this argument proves the following general result.

\begin{Thm}
Any \textup(pre-\textup)Fuchsian operator is $\W$-equivalent to an Euler operator from $\C[\eu]$.
\end{Thm}

\begin{proof}
Let, as before, $J$ denote the Euler-companion matrix of $n$ linear independent solutions $u_1,\dots,u_n$ of the equation $Lu=0$: unlike the usual companion matrix, it is obtained by applying the iterated Euler derivation $\eu$ instead of $\d$ to the functions $u_i$:
\begin{equation*}
J=J(t)=\begin{pmatrix}1 \\ \eu \\ \vdots \\ \eu^{n-1}\end{pmatrix}
\cdot\begin{pmatrix} u_1&u_2&\hdots& u_n\end{pmatrix}
=
\begin{pmatrix}
u_1&u_2&\hdots&u_n
\\
\eu u_1&\eu u_2&\hdots&\eu u_n
\\
\vdots&\vdots&\ddots&\vdots
\\
\eu^{n-1}u_1&\eu^{n-1}u_2&\hdots&\eu^{n-1}u_n
\end{pmatrix}
\end{equation*}
Unlike in the nonsingular case, we cannot guarantee anymore that $J(t)$ is holomorphic and invertible: its entries are in general multivalued and grow moderately at the origin. The companion matrix $J(t)$ has a monodromy factor  $C\in\GL(n,\C)$: $\Delta J(t)=J(t)C$ exactly as in \eqref{monodromy} which applies to each row of the matrix $J$. Yet one can always find an Euler equation whose tuple of solutions $v=(v_1,\cdots,v_n)$ will exhibit exactly the same monodromy matrix factor $C$, see \cite{thebook}*{Proposition 19.29}: $\Delta v=vC$. The corresponding linear system of algebraic equations takes the form
\begin{equation*}
\begin{pmatrix}
h_{n}&h_{n-1}&\cdots&h_1
\end{pmatrix} J(t)=\begin{pmatrix}
v_1&v_2&\cdots&v_n
\end{pmatrix}
\end{equation*}
The solution is given by the product $\begin{pmatrix}
h_{n}&h_{n-1}&\cdots&h_1
\end{pmatrix}=\begin{pmatrix}
v_1&v_2&\cdots&v_n
\end{pmatrix}\cdot J^{-1}(t)$. This product is single-valued: after analytic continuation around the origin we have
\begin{equation*}
\Delta \begin{pmatrix}
h_{n}&\cdots&h_1
\end{pmatrix}=\begin{pmatrix}
v_1&\cdots&v_n
\end{pmatrix}C\cdot C^{-1}J^{-1}(t)=\begin{pmatrix}
h_{n}&\cdots&h_1
\end{pmatrix}.
\end{equation*}
Because of the moderate growth assumption, the coefficients $h_j$ of the conjugating operator $H$ must be meromorphic germs at the origin, thus $H\in\W$.
\end{proof}

\begin{Rem}\begin{small}
There is no reason to expect that the operator $H$ conjugating two Fuchsian operators is necessarily (pre)-Fuchsian. Indeed, let $L$ be any Fuchsian operator and $H$ an irregular conjugacy. Applying $H$ to any basis tuple of solutions $u_1,\dots,u_n$ of $Lu=0$, we obtain another tuple $v_i=Hu_i$ which also grow moderately and have the same monodromy as $u_i$. By the Fuchs theorem, they satisfy a Fuchsian equation $Mv=0$. Thus, the two Fuchsian operators $L,M$ are conjugated by a (unique for the reasons of order/dimension) irregular operator $H$.\par\end{small}
\end{Rem}

In other words, the (general) Weyl classification of (pre)-Fuchsian operators coincides with the classification of their monodromy matrices, very much like the meromorphic gauge classification of linear systems \eqref{ls}.

%\begin{Prob}
%The Weyl classification of the whole algebra $\W$ is not discussed here. Actually, it looks now a very interesting problem, I am pretty much sure that its solution can be extracted from \cite{ore}.
%\end{Prob}

\section{Fuchsian equivalence}

It appears that a comprehensive analog of the holomorphic gauge equivalence between Fuchsian linear systems is the Fuchsian equivalence of Fuchsian operators: modulo technical details, this equivalence means the Weyl conjugacy \eqref{st} by a Fuchsian operator $H$ subject to certain nondegeneracy constraints. We start with developing the \emph{formal theory} of such equivalence via \emph{noncommutative formal power series}.

\subsection{Noncommutative Taylor expansions for Fuchsian operators}
Together with the representation of differential operators from the ring $\W=\M[\eu]$ as polynomials in $\eu\in\W$ with coefficients in $\M$, we can expand them in convergent \emph{noncommutative} Laurent series in the variable $t\in (\C^1,0)$ with (right) coefficients from the (commutative) ring $\C[\eu]$. Any operator $L\in\W$ of order $n=\ord L$ can be expanded under the form
$$
 L=\sum_{k=-N}^{+\infty}t^kp_k(\eu),\qquad \max_k\deg_{\eu} p_k=n, \quad N<+\infty.
$$
The operator is Fuchsian if and only if all powers are nonnegative and the leading coefficient $p_0$ is of the maximal degree: $L\in\F$ if and only if
\begin{equation}\label{ts}
L=\sum_{k=0}^\infty t^k\,p_k(\eu),\qquad p_k\in\C[\eu],\ \deg p_k\le n,\qquad  \deg p_0=n.
\end{equation}
The differential operator $p_0\in\C[\eu]\subset\F$ is called the Euler part of $L$, or its \emph{Eulerization} (in analogy with linearization) and denoted by $\E(L)$.

Very informally, an operator with holomorphic coefficients can be considered as a small perturbation of its Eulerization. The Fuchsian condition means that this perturbation is nonsingular, i.e., it does not increase the order of the Euler part, in the same way as the nonsingularity condition means that the operator can be considered as a small nonsingular perturbation of the operator $p(\d)$.

The key tool used in this paper will be a systematic use of the Taylor expansion \eqref{ts}  in exactly the same way the theory of formal series with matrix coefficients of the form $H(t)=\sum_{k=0}^\infty t^kH_k$, $H\in\Mat(n,\C)$, is used in the theory of formal normal forms of vector fields \cite{thebook}*{\parasymbol 4 and \parasymbol 16}. Note the difference in the algebraic nature of the noncommutativity: in the matrix case the coefficients $H_k$ commute with the variable $t$ but in general do not commute between themselves. In the operator case the polynomial coefficients $p_k\in\C[\eu]$ commute between themselves but do not commute with $t$.

Together with the convergent noncommutative Taylor series, it is convenient to introduce the class of \emph{formal} Fuchsian operators.
\begin{Def}
A formal Fuchsian operator is a formal series of the form \eqref{ts} without any convergence assumption. The set of formal Fuchsian operators is denoted by $\^\F$.
\end{Def}

\begin{Rem}
For any two Fuchsian operators $L,M\in\F$ their composition is again a Fuchsian operator of order $\ord L+\ord M$. If $\ord M\le \ord L$, then the incomplete ratio $Q$ as in \eqref{divrem} is a Fuchsian operator of order $\ord L-\ord M$. The same applies to $\^\F$. This follows from direct inspection of the division algorithm.

However, the set $\F$ is \emph{not} a subalgebra of $\W$: the sum of two Fuchsian operators may well be non-Fuchsian. Hence the remainder $R$ as in \eqref{divrem} after the incomplete division may well turn non-Fuchsian (the leading coefficient may vanish). Yet for any two given Fuchsian operators $L,M$ of degrees $n>m$ one can construct a relaxed division with remainder $L=Q'M+R'$ with $\ord Q'=n-m$ and $\ord R'=m$ and $Q,R$ Fuchsian. Indeed, it suffices to modify the standard division with remainder $L=QM+R$ with $\ord R\le n-1$ (assuming $Q,R$ with holomorphic coefficients) and replace $Q'=Q-1$, $R'=M+R$: the latter operators will be automatically Fuchsian.
\end{Rem}

\subsection{Main definition}
\begin{Def}\label{def:fuchseq}
Two operators $L,M\in\W$ of the same order $n$ are called \emph{Fuchsian equivalent} (or $\F$-equivalent), if there exist two Fuchsian operators $H,K\in\F$ such that $MH=KL$  (exactly as in Definition~\ref{def:we}), but with the additional property that the Euler parts of $H$ and $L$ are mutually prime (have no common roots), i.e., $\gcd(\E(H),\E(L))=1\in\C[\eu]$.

Two formal Fuchsian operators $L,M\in\^\F$ are called \emph{formally $\F$-equivalent} ($\^\F$-equivalent in short) if there exist $H,K\in\^\F$ such that $MH=KL$ and the Euler parts of $H,L$ are mutually prime.
\end{Def}

We expect that the Fuchsian classification (and its formal counterpart) for arbitrary operators from $\W$ will be a very challenging problem with the Stokes phenomenon \cite{thebook}*{\parasymbol 20} manifesting itself in a new way. However, \emph{everywhere below we will deal only with the $\F$-equivalence between \emph{Fuchsian} operators}.

Note that we dropped the condition on the order of $H,K$ which can now be \emph{higher} than $n$. Besides, in this definition we replaced the condition $\gcd(H,L)=1\in\W$ from \eqref{st} by the \emph{stronger} condition on the mutual primality of the Eulerizations.

%\begin{Rem}
%We may relax the condition of convergence of the series \eqref{ts} and consider \emph{formal Fuchsian operators} without any assumptions on the convergence of the series. On this way we can construct the formal algebras $\^\W$ of differential operators with formal Laurent coefficients and its subset $\^\F$ of formal Fuchsian operators. The definitions of $\W$-, resp., $\F$-conjugacy can be modified accordingly to produce their formal analogs.
%
%The difference between the convergent and the formal cases is known to be crucial for irregular (non-Fuchsian) operators. However, for Fuchsian operators there is no difference between $\F$- and $\^\F$-classification. This will be explained in \secref{sec:convergence}. Meanwhile we will formulate the results in the analytic case and give only formal proofs.
%\end{Rem}

\begin{Thm}\label{thm:F-equ}
$\^\F$-conjugacy is indeed an equivalence relation\textup: it is reflexive, symmetric and transitive.
\end{Thm}

Reflexivity is obvious: each operator $L$ is $\F$-equivalent to itself by admissible conjugacy $H=1$ (which is a zero order Fuchsian operator).

The transitivity is even simpler compared to the proof of Theorem~\ref{thm:we}: we do not  replace the composition $H_2H_1$ of $\F$-conjugacies, which is always Fuchsian, by its remainder $\bmod L_1$, which may be non-Fuchsian.

However, the proof of the symmetry, given in Lemma~\ref{lem:berest} relies on the possibility of representing the identical operator $1$ by a combination $1=UL+VH$ with Fuchsian operators $U,V\in\^\F$. Simple example shows that even under the stronger assumption $\gcd(\E(L),E(H))=1$, this representation is not always possible with operators of the minimal order $n-1$.

To correct the situation, one has to allow operators of above-the-minimal order.

\subsection{Fuchsian invertibility}
It will be convenient to introduce the following notation:
\begin{equation}\label{gcd0}
 \forall L,H\in\F\quad \gcd\nolimits_0(L,H)=\gcd(\E(L),\E(H))\in\C[\eu].
\end{equation}
Using this notation, the second condition of $\F$-equivalence can be shortened to $\gcd_0(L,H)=1$.

As follows from the proof of Lemma~\ref{lem:berest}, the key step is to show that if $H$ is a Fuchsian operator such that $\gcd_0(L,H)=1$, then there exist two Fuchsian operators $U,V\in\F$ such that $UL+VH=1\in\F$ and $\gcd_0(V,L)=1$.
Recall that if $p,q\in\C[\eu]$ are two relatively prime polynomials of respective degrees $n,m$, then the linear \emph{Sylvester map} from $\C^m\times\C^n$ to $\C^{m+n}$
\begin{equation}\label{sylv}
\boldsymbol S=\boldsymbol S_{p,q}\:(u,v)\mapsto pu+qv,\qquad  \deg u\le m-1,\ \deg v\le n-1,
\end{equation}
is injective and surjective (here we identify $\C^m$ and $\C^{n}$ with the linear spaces of polynomials of degree $\le m-1$, resp., $\le n-1$). In particular, any equation in $\C[\eu]$ of the form
$$
 up+vq=r,\quad \deg r\le \deg p+\deg q-1,
$$
is solvable with respect to $u,v$ constrained as above.

The following result is the analog of the implicit function theorem for differential operators.

\begin{Lem}\label{lem:tayl}
If $L,M\in\F$ are two Fuchsian operators with $\gcd_0(L,M)=1$, then for any operator $R=\sum t^k r_k(\eu)$ of order $\le \ord L+\ord M-1$ with holomorphic coefficients the equation $$UL+VM=R$$ is solvable with respect to the operators $U,V$ of orders $\ord M-1$ and $\ord L-1$ respectively, also with holomorphic coefficients.
\end{Lem}

Note that we do not assume $R$ Fuchsian, nor claim the Fuchsianity of $U$ and $V$.

\begin{proof}
The proof is achieved by inductive determination of the coefficients of the unknown operators $U,V$.

Substitute the expansions for $L=\sum_0^\infty t^k p_k$ and $M=\sum_0^\infty t^kq_k$ and the unknown operators $U=\sum_{0}^\infty t^k u_k$, $V=\sum_0^\infty t^k v_k$,   $p_k,q_k,u_k,v_k\in\C[\eu]$ into the equation $UL+VM=R$:
\begin{multline*}
 (u_0+tu_1+t^2u_2+\cdots)(p_0+tp_1+t^2p_2+\cdots)\\+(v_0+tv_1+\cdots)(q_0+tq_1+\cdots)=r_0+tr_1+t^2r_2\cdots.
\end{multline*}
Using the commutation rules \eqref{shift}, we reduce this operator identity to an infinite series of identities in $\C[\eu]$,
$$
\begin{aligned}
u_0p_0+v_0q_0&=r_0,
\\
u_0\sh 1 p_1+u_1p_0+v_0\sh 1 q_1+v_1q_0&=r_1,
\\
u_0\sh 2p_2+u_1\sh 1p_1+u_2p_0+v_0\sh 2q_2+v_1\sh 1q_1+v_2q_0&=r_2,
\\
{\makebox[0.3\columnwidth]{\dotfill}}
\\
\cdots+u_k p_0+v_kq_0&=r_k,\qquad \forall k\ge 0.
\end{aligned}
$$
This system has a ``triangular'' form: each left hand side is the sum of the term $u_kp_0+v_kq_0=\boldsymbol S(u_k,v_k)$ and terms involved shifted polynomials $u_i\sh j$, $v_i\sh j$ with $i,j<k$. By the relative primality of $p_0,q_0$, for any combination of previously defined coefficients the equation number $k$ is always \emph{uniquely} solvable with respect to some polynomials $\deg u_k\le \deg q_0-1$, $\deg v_k\le \deg p_0-1$.
\end{proof}

\begin{Rem}\begin{small}
The proof of the convergence of the series for $U$ and $V$ can be obtained directly by control over the growth of the polynomial coefficients.

However, a simpler argument works. Expanding $U,V$ as polynomials of $\eu$ with analytic coefficients from $\O(\C,0)$, $$
U=\sum_k a_k(t)\eu^k,\quad V=\sum_j b_j(t)\eu^j,
$$
we see that the operator equation $UL+VM=R$ reduces to a system of linear nonhomogeneous algebraic equations with respect to the unknown coefficients $a(t),b(t)$: in a symbolic way, this system can be written as $C(t)z=f(t)$, where $C(t)$ is an $(n+m)\times(n+m)$-matrix with holomorphic entries (produced from the coefficients of the operators $L$ and $M$ and their $\eu$-derivatives), and $f(t)$ is an $(n+m)$-dimensional holomorphic vector function.

One can easily see that the condition $\gcd_0(L,M)=1$ implies that the matrix $C(0)$ is nondegenerate and the system has a holomorphic solution. The formal computation amounts to the formal inversion of the corresponding matrix $C(t)$ without even explicitly writing it down.\par\end{small}
\end{Rem}

Unfortunately, the goal of solving the equation $UL+VH=1$ in the class of Fuchsian operators cannot be achieved using only this Lemma: indeed, there is no way to ensure that the polynomial $v_0=\E(V)$ has the maximal degree equal to $\ord V$. The way out is to look for a solution of higher order.

We look for a Fuchsian solution of the equation $UL+VH=1$ in the class of operators $\ord U\le \ord H=m$, $\ord V\le \ord L=n$ as follows,
$$
 U=H+U_{m-1},\ V=-L+V_{n-1},\qquad \ord U_{m-1}\le m-1,\ \ord V_{n-1}\le n-1.
$$
Substituting these formulas into the original equation, we transform it to the equation
\begin{equation}\label{comm}
 U_{m-1}L+V_{n-1}H=1-[H,L], \qquad [H,L]=HL-LH.
\end{equation}
The commutator $[L,H]$ of the two Fuchsian operators possesses two obvious properties. It is an operator of order no greater than $\ord L+\ord H-1$ (the highest order terms, in the expansion \eqref{ff}, the \emph{symbols} of operators cancel each other when computing the commutator). On the other hand, its Euler part vanishes.

Thus the equation is solvable by virtue of Lemma~\ref{lem:tayl}, and
$$
\E(U_{m-1})\E(L)+\E(V_{n-1})\E(H)=1\in\C[\eu].
$$
In other words, $\gcd_0(V_{n-1},L)=1$. The operator $V=-L+V_{n-1}$ is Fuchsian (since $L$ is Fuchsian of order $n$), and $\gcd_0(V,L)=\gcd_0(V_{n-1},L)=1$.

This completes the proof of the symmetry of the $\F$-equivalence.

%\textcolor{red}\bgroup
%\subsection{Why Fuchsian?}
%The choice of the class of Fuchsian operators for conjugacy is justified by the additional structure on the space of
%solutions of Fuchsian equations, which is preserved by Fuchsian operators.
%
%Consider the (infinite-dimensional) space spanned by quasi-monomials of the form $u_{\l k}(t)=t^\l \ln^{k-1}t$, $\l\in\C$, $k\in\N$. This space has a natural \emph{partial} order:
%\begin{equation}\label{po}
% u_{\l k}\prec u_{\mu l}\iff \mu-\l\in\N\quad\text{or}\quad\mu=\l,\ k>l.
%\end{equation}
%It is well known \cite{ince}*{Chapter XVI} that any Fuchsian equation $Lu=0$, $L=p_0+tp_1+\cdots$ admits a fundamental system of \emph{distinguished} solutions of the form
%\begin{equation}\label{frob}
% u(t)=u_{\l k}+\sum_{(\mu l)\succ (\l k) }c_{\mu l}u_{\mu l},\quad p_0(\l)=0,\ k\le \operatorname{mult}_\l p_0,
%\end{equation}
%where $\l$ ranges over the set of roots of $p_0$ and $k$ is at most the multiplicity of the corresponding root (thus $k=1$ if $\l$ is a simple root of $p_0$, so the corresponding leading term involves no logarithms).
%
%\begin{Prop}
%Any Fuchsian conjugacy preserves the leading term of each distinguished solution up to a coefficient.
%\end{Prop}
%\egroup

\section{Formal $\F$-classification of Fuchsian operators}

This and the next section contain the main results of the paper. They are established on the formal level, yet at the end we will show that any $\^\F$-conjugacy between convergent Fuchsian operators in fact converges.

\subsection{Nonresonant case: Eulerization}
We start by establishing an analog of the linearization theorem for nonresonant systems, cf.~with the second line in Table~1.

\begin{Def}\label{def:resonance}
A Fuchsian operator $L\in\F$ is nonresonant, if no two roots of $\E(L)\in\C[\eu]$ differ by a \emph{positive} integer number (multiple roots are allowed).

%We say that $L$ exhibits a resonance of order $k$ with multiplicity $\nu\ge 1$, if $\deg\gcd(p_0,p_0\sh k)=\nu$ (simple resonance if $\nu=1$).
\end{Def}

\begin{Prop}\label{prop:nonres}
A nonresonant Fuchsian operator is $\F$-equivalent to its Euler part.
\end{Prop}

\begin{proof}
Consider the expansion of the operator: $L=\sum_{j=0}^\infty t^jp_j(\eu)$, $p_0=\E(L)$. We look for an operator $H=\sum t^j h_j(\eu)$ which would solve (together with some other Fuchsian operator $K=\sum t^j k_j(\eu)\in\F$) the operator equation $p_0(\eu)H=KL$. After substituting the expansions and using the commutation rule \eqref{euler-der}, we obtain in the left hand side the operator
$$
p_0(\eu)H=p_0 h_0+tp_0\sh 1h_1+\cdots+t^jp_0\sh j h_j+\cdots,
$$
cf.~with the notation \eqref{shift}--\eqref{shift-1}. In the right hand side the expansion for
$$
 KL=(k_0+tk_1+t^2k_2+\cdots)(p_0+tp_1+t^2p_2+\cdots)
$$
will have more complicated form: the term proportional to $t^j$ has the form
$$
 t^j(k_jp_0+k_{j-1}\sh 1p_1+k_{j-2}\sh 2p_2+\cdots+k_0\sh j p_j).
$$
The operator equation thus splits into an infinite number of polynomial equations involving the known polynomials $p_j$ and unknown $h_j,k_j$ as follows,
\begin{equation}\label{triangle}
\begin{aligned}
 p_0h_0&=k_0p_0,\\
 p_0\sh 1 h_1&=k_1p_0+k_0\sh 1p_1,\\
 p_0\sh 2h_2&=k_2p_0+k_1\sh 1p_1+k_0\sh 2p_2,\\
 &\makebox[0.3\columnwidth]{\dotfill}\\
 p_0\sh jh_j&=k_jp_0+k_{j-1}\sh 1p_1+\cdots+k_0\sh jp_j,\\
 &\makebox[0.3\columnwidth]{\dotfill}
\end{aligned}
\end{equation}
This system can be solved inductively: on the first step we choose $h_0=k_0$ any polynomial of degree $n-1$ relatively prime with $p_0$. The remaining equations all have the common structure:
\begin{equation}\label{lze}
 p_0\sh jh_j-p_0 k_j=u_j,
\end{equation}
where $u_j\in\C[\eu]$ is a polynomial of degree $\le 2n-1$ built from the already obtained polynomials $k_0,\dots,k_{j-1}$ and known $p_1,\dots,p_j$.

If $L$ is nonresonant, no two roots of $p_0$ differ by a positive integer $j$, hence  $\gcd(p_0,p_0\sh j)=1$ for all $j=1,2,\dots$ and any such equation is (uniquely) solvable by a suitable pair $(h_j,k_j)$ of polynomials of degree $\le n-1$. Thus the entire infinite system admits a \emph{formal} solution $(H,K)$.

It remains to show that if the series for $L=\sum t^jp_j$ was convergent, so will be the series for $L$ and $K$. This can be done by the direct estimates, yet we give a general proof avoiding all computations later, in \secref{sec:convergence}.
\end{proof}

\subsection{$\F$-normal form and apparent singularities}
Some properties of solutions can be easily described in terms of $\F$-equivalence. Recall that a singular point of a differential equation is called \emph{apparent}, if all solutions of this equation are holomorphic at this point.

\begin{Prop}
A Fuchsian operator has only meromorphic solutions if and only if it is $\F$-equivalent to an Euler operator $L=\E(L)=p_0(\eu)$ with integer pairwise different roots, $p_0(\eu)=\prod_{i=1}^n (\eu-\l_i)$, $\l_i\in\Z$, $\l_i\ne\l_k$ for $i\ne k$.

A Fuchsian operator has only holomorphic solutions, if and only if it is $\F$-equivalent to an Euler operator as above, with nonnegative pairwise distinct roots, $\l_i\in\Z_+$.
\end{Prop}

\begin{proof}
In one direction both statements are obvious. We show that Fuchsian operators with only meromorphic (resp., holomorphic) solutions are $\F$-equivalent to an Euler equation as above.

One can easily show that any $n$-dimensional $\C$-linear subspace $\ell$ in $\M(\C,0)$ (resp., in $\O(\C,0)$) admits a basis of the germs of the form $f_i=t^{\l_i}u_i(t)$ with \emph{pairwise different} integer (resp., nonnegative integer) powers $\l_i$, $u_i\in\O(\C,0)$ and $u_i(0)=1$. Indeed, we can start with \emph{any} $\C$-basis $f_1,\dots,f_n$ in $\ell$ and normalize them so that each function has a monic leading term $t^{\l_i}(1+\cdots)$. If there are two equal powers among the initial collection, $\l_i=\l_k$, then their difference (which cannot be identically zero by linear independence) has the leading term proportional to $t^{\mu}$, $\l_i=\l_k<\mu\in\Z$. Repeating this procedure finitely many steps, one can always achive the situation when $\l_i\ne\l_k$.

Now we construct explicitly the Fuchsian operator $H=\sum t^jh_j(\eu)$ which would transform the monomials $t^{\l_i}$, $i=1,\dots,n$, to the functions $c_if_i$ for suitable coefficients $c_i\in\C$. Note that each monomial $t^{\l_i}$ is an eigenfunction for any Euler operator, in particular, $h_j(\eu)t^{\l_i}=h_j(\l_i)t^{\l_i}$, and therefore
$$
 Ht^{\l_i}=\f_i(t)t^{\l_i},\qquad \f_i(t)=\sum_{j\ge 0} t^jh_j(\l_i).
$$
The equations $Ht^{\l_i}=t^{\l_i}(c_i+c_{i1}t+c_{i2}t^2+\cdots)$ are thus transformed to the infinite number of interpolation problems,
$$
 h_0(\l_i)=c_i, \qquad h_j(\l_i)=c_{ij},\qquad i=1,\dots,n,\quad j=1,2,\dots
$$
Such problems are always solvable by polynomials $h_j\in\C[\eu]$ of degree $\le n-1$, and since $c_i=h_0(\l_i)\ne 0$, we have $\gcd(h_0,p_0)=1$. By a suitable (generic) choice of the constants $c_i\ne0$, one may guarantee that $\deg h_0=n-1$, that is, $H$ is indeed a Fuchsian operator, as required for the $\F$-equivalence.
\end{proof}

Note that in both cases the normal form is maximally resonant: all differences between the roots of the Euler part are integer.

\begin{Rem}
This results shows to what extent the $\F$-equivalence is more fine than the $\W$-equivalence. Indeed, given the trivial monodromy, all operators having only meromorphic solutions, are $\W$-equivalent to the same Euler operator $t^{-n}\d^n=\eu(\eu-1)\cdots(\eu-n+1)$. On the other hand, two different Euler operators are never $\F$-equivalent: if $\gcd (p_0,h_0)=1$, then the identity $p_0h_0=q_0k_0$ in $\C[\eu]$, the first line from \eqref{triangle}, implies that $p_0=q_0$ and $h_0=k_0$.
\end{Rem}

\subsection{Resonant case: Homological equation and its solvability}
If some of the roots of the Euler part $p_0$ differ by a natural number, then the corresponding equations \eqref{lze} may become unsolvable and in general transforming a resonant Fuchsian operator $L\in\F$ to its Euler part $\E(L)\in\C[\eu]$ is impossible. However, one can use $\F$-equivalence to \emph{simplify} Fuchsian operators.

If a Fuchsian operator $H=\sum t^jh_j(\eu)$ conjugates $L$ with another operator $M=\sum t^j q_j(\eu)\in\F$, then the left hand side of the identity $p_0(\eu)H=KL$ should be replaced by
\begin{multline}
 MH=(p_0+tq_1+t^2q_2+\cdots)(h_0+th_1+t^2h_2+\cdots)\\
 =p_0h_0+t(q_1h_0+p_0\sh 1h_1)+\cdots+\\
 t^j(q_j h_0+q_{j-1}h_1\sh 1+\cdots+p_0h_j\sh j)+ \cdots,
\end{multline}
and, accordingly, the equations  \eqref{lze} should be replaced by the equations
\begin{equation}\label{homol}
 p_0\sh jh_j-p_0 k_j+q_j h_0= v_j,\qquad j=1,2,\dots,
\end{equation}
where, as before, $v_j\in\C[\eu]$ is a polynomial of degree $\le 2n-1$ formed by (eventually shifted) combinations of $q_i,h_i,k_i$ with smaller indices $0<i<j$ and $p_1,\dots,p_j$.

First, we use the fact that although some of the equations \eqref{homol} may be non-solvable, they are always solvable for sufficiently large orders.

\begin{Prop}\label{prop:poly}
Let $L=p_0+tp_1+\cdots\in\F$ be a Fuchsian operator and $N$ the maximal natural difference between the roots of $p_0=\C_n[\eu]$.

Then $L$ is $\F$-equivalent to the \emph{polynomial} operator $M$ obtained by truncation of the Taylor series at the order $N$,
$$
  M=\sum_{j=0}^N t^jp_j(\eu)=\sum_{k=0}^n b_k(t)\eu^{n-k}\in \C[t,\eu]
$$
with polynomial coefficients $b_k\in\C[t]$ of degree $\deg_t b_k\le N$, obtained by truncation of the analytic coefficients $a_k\in\O(\C,0)$ of the initial operator $L$ at the order $N$.
\end{Prop}

\begin{proof}
First we find a pair of operators $H_0,K_0\in\W$ of order $n-1$ with holomorphic coefficients, which almost conjugate $L$ with $M$ in the form $H_0=1+\sum_{j>N}t^j h_j$, $K=1+\sum_{j>N} t^j k_j$, so that $MH_0=K_0M$. Substituting these expansions in the equations \eqref{homol}, we see that all equations of order $j=0,1,\dots,N$ are satisfied automatically if we set $q_j=p_j$ and $0=h_j=k_j$ for all $j=1,\dots,N$.

The operators $H_0,K_0$ are (usually) non-Fuchsian, since $0=\ord h_0<\ord H_0=n-1$. However, the operators $H=H_0+L$ and $K=K_0+M$ are Fuchsian, satisfy the identity $MH=M(H_0+L)=K_0L+ML=(K_0+M)L=KL$ and the nondegeneracy condition $\gcd_0(L,H)=\gcd(p_0,p_0+1)=1$ is satisfied.
\end{proof}

\subsection{Integrable normal form}
The polynomial normal form established in Proposition~\ref{prop:poly}, lacks any integrability properties. Yet using the same method, one can construct a Liouville integrable $\F$-normal form for any Fuchsian operator.

\begin{Prop}\label{prop:fnf}
A Fuchsian operator $=p_0+tp_1+\cdots\in\F$ with the Eulerization $p_0\in\C_n[\eu]$ as above, is $\F$-equivalent to an operator $M\in\F$ of the form
\begin{equation}\label{fnf}
 M=(\eu-\l_1+r_{1})\cdots(\eu-\l_n+r_n),\qquad r_{i}=r_i(t)\in\C[t],\ r_{i}(0)=0.
\end{equation}
In other words, $M$ is a \textup(noncommutative\textup) product of polynomial operators of order $1$. The degrees of the polynomials $r_i(t)$ are explicitly bounded, $\deg_t r_i(t)\le N$, where $N$, as before, is the maximal order of resonance between roots of $p_0$.
\end{Prop}

\begin{Rem}
If $\l_{i-1}=\l_i$ is a multiple root of $p_0$, still the polynomials $r_{i-1}$ and $r_i$ in general will be different.
\end{Rem}

\begin{Lem}\label{lem:factor}
Any analytic Fuchsian operator $L\in\F$ can be factorized as
\begin{equation}\label{genuine-factor}
L=(\eu-\l_1+R_1)\cdots(\eu-\l_n+R_n), \qquad R_i=R_i(t)\in\O(\C^1,0),\ R_i(0)=0,
\end{equation}
with \emph{analytic} \textup(rather than polynomial\textup) functions $R_1,\dots,R_n$.
\end{Lem}

\begin{proof}[Proof of the Lemma]
Consider an eigenfunction $u(t)$ of the monodromy operator, associated with the equation $Lu=0$: the corresponding eigenvalue is nonzero, hence $\Delta u=\mathrm e^{2\pi i\l}u$ for some $\l\in\C$. Then $u=t^{\l}v(t)$, where $v$ is a meromorphic germ, and modulo replacing $\l$ by $\l+j$ for some $j\in\Z$, we may assume that $v$ is holomorphic invertible, $v\in\O(\C,0)$, $v(0)\ne0$. Applying $\eu$ to this function, we see that $\eu u=\l t^{\l}v+t^{\l}(\eu v)=\l u+Ru$, $R=\frac{\eu v}v\in\O(\C,0)$, in other words, $u$ satisfies a Fuchsian equation of the first order and $L$ is divisible from the right by $\eu-\l-R(t)$. The quotient is again a Fuchsian operator of order $n-1$, and the process can be continued by induction.
\end{proof}

\begin{proof}[Proof of the Proposition~\ref{prop:fnf}]
Consider the factorization \eqref{genuine-factor} of the operator $L$ as provided by Lemma~\ref{lem:factor}, and replace each analytic function $R_i$ by its polynomial truncation $r_i$ to order $N$, so that $\ord_{t=0} (r_i-R_i)>N$. The (polynomial) operator $M$ thus obtained has the same $N$-jet with respect to $t$ as the initial operator $L$. By Proposition~\ref{prop:poly}, $M$ is $\F$-equivalent to $L$.
\end{proof}

The normal form established by Proposition~\ref{prop:fnf} has an advantage of being Liouville integrable. Each linear  equation of the first order is explicitly solvable ``in quadratures''. In particular, the homogeneous equation
$$
 Lu=0,\quad L=\eu-\l+r(t),\qquad r\in\C[t],\ r(0)=0,
$$
has a 1-dimensional space of solutions $u(t)=Ct^\l\exp \rho(t)$, where $\rho(t)=-\int \frac{r(t)}t \,\mathrm dt$ is a polynomial in $t$.

To solve the nonhomogeneous equations, the method of variation of constants can be used to produce a particular solution using operations of integration (computation of the primitive), exponentiation and the field operations in the field $\C(t)$ of rational functions (the details are left to the reader). Iterating this computation, one can find a general solution of the equation $Mu=0$ with a completely reduced operator $M$ as in \eqref{fnf}: if $M=L_1L_2\cdots L_n$, $\ord L_i=1$, then solution of the equation $Mu=0$ amounts to solving a chain of equations of order $1$,
\begin{equation}\label{chain-liou}
L_1 u_1=0,\ L_2 u_2=u_1\ ,\dots,\ L_nu_n=u_{n-1},\qquad u=u_n.
\end{equation}

\begin{Cor}
Any Fuchsian operator is $\^\F$-equivalent to a Liouville integrable operator.\qed
\end{Cor}

\subsection{Non-Eulerizability of resonant Fuchsian equations}
The explicit integrability of the factorized equations allows to show that the resonant Fuchsian equations, ``as a rule'', are even not $\W$-equivalent to their Euler part.

\begin{Ex}
Consider the Fuchsian operator
$$
L=(\eu+t)(\eu-1)=\E(L)+t(\eu-1)\in\F,\qquad \E(L)=\eu(\eu-1).
$$
The Euler part of $L$ has simple integer roots, hence the trivial monodromy. On the other hand, the equation $Lu=0$ can be explicitly solved. One solution, $u_1(t)=t$, satisfying the equation $(\eu-1)u=0$, is obvious. The equation $(\eu +t)v=0$ has solution $v(t)=\mathrm e^{-t}$, and another solution $u_2(t)$ of the linear non-homogeneous equation $(\eu-1)u(t)=\mathrm e^{-t}$, can be found by the method of variation of constants, $\displaystyle u_2(t)=t\int\mathrm e^{-t}t^{-2}\,\mathrm dt$. The monodromy transformation of the pair of solutions $(u_1,u_2)$ is given by the non-identical matrix
$\begin{pmatrix}
1 & 2\pi\mathrm i
\\
 & 1
\end{pmatrix}$. This means that the full operator is even not $\W$-equivalent to its Euler part.
\end{Ex}

\section{Minimal normal form}

The polynomial normal forms established in the preceding section are of rather limited interest: indeed, no attempt was made to modify the lower order terms of the Taylor expansion of the resonant Fuchsian operators.

The system of equations \eqref{homol} can be solved recursively with respect to $h_j,k_j$  even in the resonant case $\gcd(p_0,p_0\sh j)\ne 1$, provided that $q_j$ are chosen \emph{in a suitable way}: the difference $v_j-q_jh_0$ should belong to the image of the Sylvester map $\boldsymbol S_j=\boldsymbol S_{p_0,p_0\sh j}$, cf.~with \eqref{sylv}. This image consists of all polynomials of degree $\le 2n-1$ divisible by $w_j=\gcd(p_0,p_0\sh j)\in\C[\eu]$. In this section we describe possible choices for the terms $q_j$.

\subsection{Abstract normal forms}
Denote by $\mathfrak P=\C[\eu]/\left<p_0\right>\simeq\C_{n-1}[\eu]$ the quotient algebra: as a $\C$-space it is $n$-dimensional and can be identified with the residues modulo $p_0$, polynomials of degree $\le n-1$.

This quotient algebra in the simplest case where all $n$ roots $\l_1,\dots,\l_n\in\C$ of $p_0$ are simple, can be identified with the $\C$-algebra of functions on $n$ points $\L=\{\l_1,\dots,\l_n\}\subseteq\C$: $\mathfrak P=\{\f:\L\to\C\}\simeq\C\times\cdots\times\C$: any such function can be represented as the restriction of a polynomial $h\in\C_{n-1}[\eu]$ of degree $\le n-1$: $h|_\L=\f$. The functions $\f_i$ equal to $1$ at one point $\l_i\in\L$ and vanishing at all other points $\l_k\ne \l_i$, form a natural basis of $\mathfrak P$.

\begin{Rem}\begin{small}
In the general case where the roots $\l_i$ may have nontrivial multiplicities $\mu_i\in\N$,
$$
 p_0(\eu)=\prod_i (\eu-\l_i)^{\mu_i},\qquad \sum_i \mu_i=\deg p_0=n,
$$
the quotient algebra $\mathfrak P$ is naturally isomorphic to the direct sum of the local algebras $J_i\simeq \C[\eu]/(\eu-\l_i)^{\mu_i+1}$ of dimension $\mu_i$: each element of $\mathfrak P=\bigoplus_i J_i$ can be identified with a \emph{multijet}, a collection of  $\mu_i$-jets (Taylor polynomials of order $\mu_i$) at the points $\l_i\in\L\subset\C$. \end{small}
\end{Rem}

For any polynomial $s\in\C[\eu]$ the multiplication by $s$ is an endomorphism of the algebra $\mathfrak P$. It is invertible (automorphism of $\mathfrak P$) if and only if $\gcd(p_0,s)=1$.

The equations \eqref{homol} induce the equations in the algebra $\mathfrak P$:
\begin{equation}\label{homol-p}
p_0\sh j h_j+q_jh_0=v_j,\qquad j=1,2,\dots
\end{equation}
They can be re-written in the operator form as
\begin{equation}\label{hom-oper}
\boldsymbol P_j h_j+\boldsymbol H q_j=v_j
\end{equation}
where $\boldsymbol P_j,\boldsymbol H$ are endomorphisms (self-maps) of $\mathfrak P$, induced by multiplication,
\begin{equation}
 \boldsymbol P_j: h\longmapsto p_0\sh jh,
\qquad\boldsymbol H:q_j\longmapsto h_0 q_j.
\end{equation}
The endomorphisms commute between themselves and $\boldsymbol H$ is invertible.

\begin{Def}\label{def:nf}
An \emph{affine normal form} for the polynomial $p_0$ is a family of subspaces $V_j\subseteq\mathfrak P$ (not necessarily subalgebras) such that $V_j$ is complementary to the image of $\boldsymbol P_j$,
\begin{equation}\label{nf}
\boldsymbol P_j\mathfrak P+V_j=\mathfrak P\qquad j=1,2,\dots.
\end{equation}
The affine normal form is \emph{minimal}, if $\dim V_j=\dim\Ker\boldsymbol P_j$.

Without loss of generality we may assume that $V_j=0$ for all sufficiently large values of $j$ (for minimal normal forms this condition is automatically satisfied).
\end{Def}

Note that the choice of an affine normal form is by no means unique: moreover, being a normal form is an open property (small perturbation of the subspaces $V_j$ does not violate the property \eqref{nf}.

These definitions are tailored to make the following statement trivial.

\begin{Thm}\label{thm:abstract}
Let $\{V_j\}$ be an abstract affine normal form for the polynomial $p_0\in\C_n[\eu]$.

Then any Fuchsian operator $L=p_0(\eu)+tp_1(\eu)+\cdots$ is $\F$-equivalent to an operator $M=p_0+\sum_{j=1}^Nt^jq_j(\eu)$ with $q_j\in V_j$.
\end{Thm}

\begin{proof}
By invertibility of $\boldsymbol H$, we have $\boldsymbol H^{-1}\mathfrak P=\mathfrak P=\boldsymbol H\mathfrak P$. By \eqref{nf}, $\boldsymbol P_j \boldsymbol H^{-1}\mathfrak P+V_j=\mathfrak P$. Applying to both parts of the latter equality the operator $\boldsymbol H$ and using the commutativity, we see that
$$
\boldsymbol P_j\mathfrak P+\boldsymbol HV_j=\mathfrak P,
$$
that is, each homological equation \eqref{hom-oper}, regardless of the right hand side $v_j$, admits a solution $h_j\in\mathfrak P$, $q_j\in V_j$. This solution generates (by definition of $\mathfrak P$) a solution $(h_j,k_j)$ of \eqref{homol}.
\end{proof}

\subsection{Minimal affine normal form}
One possibility to chose an affine normal form is to stick to the polynomials of minimal degree modulo $p_0$. Denote by $w_j\in\C[\eu]$ the greatest common divisor $w_j=\gcd(p_0,p_0\sh j)$; this is a polynomial of degree $\nu_j\le n-1$.

\begin{Prop}\label{prop:min-nf}
Any Fuchsian operator $L=p_0(\eu)+tp_1(\eu)+\cdots$ is $\F$-equivalent to a polynomial operator of the form $M=p_0+\sum_j t^j q_j(\eu)$ with $\deg q_j\le \nu_j-1$. In particular, $q_j=0$ for all nonresonant orders.
\end{Prop}

\begin{proof}
It suffices to note that the subspace $V_j\simeq\C_{\nu_j-1}[\eu]\subseteq\C_{n-1}[\eu]\simeq\mathfrak P$ is naturally complementary to the image $\boldsymbol P_j\mathfrak P$ which consists of all polynomials of degree $\le n-1$, divisible by $w_j$. This follows from the division with remainder by $w_j$ in $\mathfrak P\simeq\C_{n-1}[\eu]$.
\end{proof}

Note that the family of the subspaces $V_j\simeq\C_{\nu_j-1}[\eu]$ is a minimal normal form.

\begin{Ex}\label{ex:single}
Assume that the operator $L$ has a single resonance, i.e., only one pair of roots of $p_0$ differs by an integer $k$. Then the operator $L$ is $\F$-equivalent to $p_0(\eu)+ct^k$, $c\in\C$.
\end{Ex}

\subsection{Separation of resonances}
A different strategy of choice of the subspaces $\{V_j\}$ constituting a normal form, is to reproduce the strategy which results in the Poincar\'e--Dulac normal form for Fuchsian systems with diagonal residue matrix $A\in\Mat(n,\C)$. Recall that in this case instead of solving the homological equation \eqref{homol}, one has to solve matrix equations of the form $[A, H]+jH=B_j$, where $B_j$ are given matrices from $\Mat(n,\C)$, cf.~with \cite{thebook}*{Theorem 16.15}. The operator taking a matrix $H$ into the twisted commutator as above, is diagonal in the natural basis of matrices having only one nonzero entry, and kernel of this operator is naturally complementary to its image.

An analogous construction can be applied in the case of operators $\boldsymbol P_j$ if the polynomial $p_0=\E(L)$ has simple roots. Then multiplication by any polynomial, including $w_j$, is diagonal, hence one can choose $V_j=\operatorname{Ker}\boldsymbol P_j$. The polynomials $q_j$ which appear in the corresponding normal form, will be vanishing at all roots of $p_0/w_j$, hence divisible by the latter polynomial (recall that we consider polynomials of degree $\le \deg p_0-1$). In particular, if a certain root $\l_i$ of $p_0$ does not appear in any resonance, then all polynomials $q_j$ in the normal form will be divisible by $\eu-\l_i$, and therefore the operator $M$ in the normal form established in Theorem~\ref{thm:abstract} will be divisible (from the right) by the first order Euler operators $\eu-\l_i\in\F$.

This claim gives a partial effective factorization of the normal form \eqref{fnf}, which allows to identify factors with $r_i=0$. In the next section we explain how one can give an accurate description of the factors in \eqref{fnf} in general.

\subsection{Completely reducible minimal normal form}
Occurrence of resonances between the roots $\L=\{\l_1,\dots,\l_n\}\subset \C$ of the polynomial $p_0\in\C_n[\eu]$ allows to introduce certain combinatorial structures. First, the (natural linear) order on $\Z$ induces a partial order on the roots: $\l_i\ge \l_k \iff\l_i-\l_k\in\Z$.

\begin{Rem}\label{rem:no}
If the Euler part has multiple roots, then the set $\L$ contains repetitions. To simplify the subsequent arguments, it is convenient to extend the partial order to a full order as follows. The roots of $p_0$ are subdivided in \emph{resonant groups} in such a way that inside each group all roots have integer differences (and hence are comparable in the sense of the partial order). Different resonant groups can be arranged between themselves in any way. The corresponding order is conveniently represented by the enumeration of the set of all roots $\L=\{\l_i\}$ in the non-decreasing order. This will make $\L$ into an ordered set naturally isomorphic to $\{1,2,\dots,n\}$: multiple roots of $p_0$ occupy consecutive positions in this list. We call this order a \emph{natural order} on $\L$ (it is not unique, since different resonant groups can be transposed, but is convenient for the formulations).
\end{Rem}

Second, for each order we can list all roots which produce this resonances of this order. Given a natural index $j\in\N$, we define
\begin{equation}\label{updown}
 \L_j=\{\l\in\L: \l+j\in\L\}\subset\L,\qquad j=1,2,\dots.
\end{equation}

This definition, unambiguous in the case where $p_0$ has only simple roots, should be modified as follows: if $\mu+j=\varkappa$ are two roots in resonance and the multiplicities of $\mu,\varkappa$ in $\L$ (the list which now may have repetitions) are $m,k$ respectively, then $\mu$ enters $\L_j$ with the multiplicity equal to $\min(m,k)$, that is, with its multiplicity as the root of the polynomial $w_j=\gcd (p_0,p_0\sh {j})$.

Together with the sets $\L_j\subset\L$ it is convenient to consider also their fully ordered counterparts (cf.~with Remark~\ref{rem:no}), the sets of the corresponding indices $I_j\subset\{1,2,\dots,n\}$. In the case where $\l$ a root of  $p_0$ with multiplicity $m>1$ and of $p_0\sh j$ with multiplicity $k<m$, we include in $I_j$ \emph{the last $k$ instances} where $\l$ enters $\L$ (out of the total $m$).

The dual description can be given by the sets $J(\l)$, which for any root $\l\in\L$ consists of the natural numbers $j\in\N$ such that $\l+j\in\L$. The case of multiple roots needs no special treatment.

Recall that support (or the Newton diagram) of a polynomial $r=\sum c_k t^k\in\C[t]$ is the set of indices $k\in\N$ such that the corresponding coefficient $c_k$ is nonzero: $\supp r=\{k:c_k\ne 0\}\subset\N$.

\begin{Thm}\label{thm:min-fnf}
Any Fuchsian operator is $\F$-equivalent to a completely reducible operator of the form
\begin{equation}\label{min-fnf}
\begin{gathered}
 L=(\eu-\l_1+r_1(t))\cdots(\eu-\l_n+r_n(t)),\\ r_i\in\C[t],\quad \supp r_i\subseteq J(\l_i),\qquad i=1,\dots,n.
\end{gathered}
\end{equation}
In particular, $\deg r_i\le \max\{j\in\N:\ \l_i+j\in\L\}$.
\end{Thm}

The rest of this section contains the proof of this theorem.

\subsection{Expansion of noncommutative products}
From that moment we assume that the roots $\l_i$ are labeled in a natural order, see Remark~\ref{rem:no}.

Consider the operators $E_{ij}\in\F$ of the form \eqref{min-fnf} in the case where only one polynomial $r_i$ is different from zero and is itself a monomial of degree $j$:
$$
 E_{ij}=(\eu-\l_1)\cdots(\eu-\l_{i-1})(\eu-\l_i+t^j)(\eu-\l_{i+1})\cdots(\eu-\l_n),\qquad i=1,\dots,n.
$$
After complete expansion of $E_{ij}$ we obtain
\begin{equation}\label{pij}
\begin{gathered}
 E_{ij}=p_0(\eu)+t^j p_{ij}(\eu),\qquad p_{ij}\in\C[\eu],\ i=1,\dots,n,
  \\
 p_{ij}=(\eu-\l_1+j)\cdots(\eu-\l_{i-1}+j)(\eu-\l_{i+1})\cdots(\eu-\l_n).
 \end{gathered}
\end{equation}
In other words, $p_{ij}$ is obtained by shifting the argument by $j$ in the first $i-1$ terms of the ordered factorization of $p_0$, while keeping the last terms the same as in $p_0$. Accordingly, the roots of $p_{ij}$ are obtained by shifting the first $i-1$ roots of $\L$ \emph{to the left} by $j$ units, removing the $i$th root from the list and keeping the remaining (larger) roots in place. Speaking informally, $p_{ij}$ has a gap on $i$th place in the (partially) ordered set $\L$.

\begin{Lem}\label{lem:indep}
For each $j\ge 1$ the polynomials $p_{1j},\dots,p_{nj}\in\C_{n-1}[\eu]$ are linear independent over $\C$.
\end{Lem}
\begin{proof}
A vanishing $\C$-linear combination of polynomials $p_{1j},\dots,p_{nj}$ after division by $p_0$ would result in a vanishing $\C$-linear combination between the corresponding rational functions. However, this is impossible, since the first fraction will have a pole at $\l_1$, and in a similar way  $\frac{p_{ij}}{p_0}$  will have either a pole at $\l_i$, or (if $\l_i=\l_{i-1}$ was a multiple root of $p_0$) the order of the pole will increase compared with the previous fraction. Since the roots were ordered, this new poles appear at the points where all previous ratios were holomorphic, which means that no linear combination can arise in the process.
\end{proof}

\begin{Cor}\label{cor:epi}
For any $j$ the polynomials $p_{1j},\dots,p_{nj}$ span $\C_{n-1}[\eu]$.
\end{Cor}

\begin{proof}
Since these polynomials are linear independent, they span an $n$-dimensional $\C$-subspace in $\C_{n-1}[\eu]$ which for the reasons of dimension must coincide with $\C_{n-1}[\eu]$.
\end{proof}

A minor modification of this argument proves a similar statement.

\begin{Lem}
Let $j\in\N$ be a natural number and $w_j=\gcd(p_0,p_0\sh j)$. Then the polynomials $p_{ij}$ for $i\in I_j$ are linear independent modulo $w_j$.
\end{Lem}

Note that the polynomials $p_{ij}$ and $p_{i+1,j}$ in general are different even if $\l_i=\l_{i+1}$.

\begin{proof}
Arguing as before, consider the rational fractions $\frac{p_{ij}}{w_j}$. Since the roots of $w_j$ constitute only a proper subset of $\L$, then not all of these fractions have either a new pole at $\l_i$, or a pole of larger order. On the other hand, if $\l_i\in\L_j$, then this means that one or more (depending on multiplicity) of the larger roots when shifted by $j$ will coincide with $\l_i$ and hence create a pole of $\frac1{w_j}$ of the corresponding order. In the case where $\l_i$ is a multiple root, we have to consider the fractions $\frac{p_{ij}}{w_j}$ for $i\in I_j$

This means that in the ordered subsequence $\frac{p_{ij}}{w_j}$, $i\in I_j$, the behavior of the poles will be as before (either a new pole appears or the order of the previous pole is increased). In both cases the linear dependence is impossible.
\end{proof}

\begin{Cor}\label{cor:transverse}
The linear span $V_j$ of polynomials $\{p_{ij}:i\in I_j\}\subset\C_{n-1}[\eu]\simeq\C[\eu]\bmod p_0$ is a linear subspace transversal to the image of the operator $\boldsymbol P_j$ from \eqref{nf}, and hence these subspaces form a minimal abstract normal form in the sense of Definition~\ref{def:nf}.
\end{Cor}

\begin{proof}
This follows from the linear independence above and the fact that the number of these polynomials is equal to the codimension of the image (which consists of polynomials of degree $\le n-1$ divisible by $w_j$).
\end{proof}

\subsection{Proof of Theorem~\ref{thm:min-fnf}}
Assume (by way of induction) that the a Fuchsian operator $L\in\F$ is already shown to be $\F$-equivalent to an operator $L_{j-1}\in\F$ whose $(j-1)$-jet is as in \eqref{min-fnf}, i.e.,
\begin{equation*}
\begin{gathered}
 L_{j-1}=(\eu-\l_1+r_{1,j-1})\cdots(\eu-\l_n+r_{n,j-1})+t^jv_j(\eu)+\cdots,
 \\
 r_{i,j-1}\in\C[t],\quad \supp r_i\subseteq J(\l_i)\cap[1,j-1],\quad i=1,\dots,n.
\end{gathered}
\end{equation*}
We will show that there exists an operator $L_j$ of the same form but with $\supp r_{i,j}\in J(\l_i)\cap [1,j]$, which is $\F$-equivalent to $L_{j-1}$. Indeed, adding monomials of order $j$ to the polynomials $r_{i,j-1}$,
$$
 r_{i,j}=r_{i,j-1}+c_it^j, \qquad c_i\ne 0\iff j\in J(\l_i),\quad i=1,\dots,n
$$
will affect only terms of order $j$ and higher after the expansion: the (polynomial) coefficient $v_j$ will be replaced by $v_j+\sum c_ip_{ij}$ by definition \eqref{pij} of the polynomials $p_{ij}$. By a suitable choice of the coefficients $c_i$ for $i\in I_j$, one can bring this sum into the range of the homological operator $\boldsymbol P_j$, as follows from Corollary~\ref{cor:transverse}.

For this choice the homological equation \eqref{homol} will be solvable with respect to $h_j,k_j$ by setting $h_0=1$, $q_j=-\sum c_i p_{ij}$.
Continuing this way, we eventually reach the values of $j$ which exceed the maximal order $N$ of possible resonances. The corresponding operator $L_N$, by construction $\F$-equivalent to the initial operator $L$, is $\F$-equivalent to its product part $\prod_{i=1}^n(\eu-\l_i+r_{iN}(t))$ with $\supp r_{iN}\in J(\l_i)$ by Proposition~\ref{prop:poly}.\qed

\begin{Rem}
The same argument allows to construct an effective factorization of any Fuchsian operator. Indeed, by Corollary~\ref{cor:epi}, one can always construct a linear combination $\sum_{i=1}^n c_ip_{ij}\in\C[\eu]$ which \emph{cancels} the term $v_j$. Proceeding this way, one can construct the formal factorization $L=\prod_{i=1}^n (\eu-\l_i+\hat r_i(t))$, $\hat r_i\in\C[[t]]$. One can show that in the Fuchsian case this factorization is always converging.
\end{Rem}

\subsection{Concluding remarks}
The minimality of the normal form \eqref{min-fnf} does not imply that coefficients of the first order factors $r_1,\dots,r_n\in\C[t]$ are $\F$-invariant. Nevertheless, one can expect that for operators of sufficiently high order there will appear moduli (numeric invariants) of $\F$-classification: for holomorphic gauge classification of Fuchsian systems this was discovered by V.~Kleptsyn and B.~Rabinovich \cite{kle-rab}.

\section{Convergence of the formal series}\label{sec:convergence}
Here we prove that the formal and analytic Fuchsian classifications for Fuchsian operators coincide.

More precisely, assume that two \emph{formal} operators $H,K\in\^\F$,  $H=\sum_{k=0}^{n-1}u_k(t)\eu^k$, $K=\sum_{k=0}^{n-1}v_k(t)\eu^k$ with formal coefficients $u_k,v_k\in\C[[t]]$ conjugate two Fuchsian operators $L=\sum_{k=0}^n a_k(t)\eu^k$, $M=\sum_{k=0}^n b_k(t)\eu^k$ with \emph{analytic} coefficients $a_k, b_k\in\O(\C,0)$, $a_n(0)b_n(0)\ne0$.

\begin{Thm}
The formal series for the coefficients $u_k,v_k\in\C[[t]]$ necessarily converge, hence $H,K\in\F$.
\end{Thm}

\begin{proof}
One possibility of proving this result is to control explicitly the growth rate of the Taylor coefficients. However, a simple strategy is to use the fact that a (vector) formal Taylor series which solves a holomorphic Fuchsian system of equations, is necessarily convergent.

The conjugacy equation $MH=KL$ takes the form of a noncommutative identity
\begin{equation}\label{conj-e}
\biggl(\sum_{k=0}^n b_k(t)\eu^k\biggr)\biggl(\sum_{k=0}^{n-1}u_k(t)\eu^k\biggr)=
\biggl(\sum_{k=0}^{n-1}v_k(t)\eu^k\biggr)\biggl(\sum_{k=0}^n a_k(t)\eu^k\biggr).
\end{equation}

We claim that this identity implies that the coefficients $u_k(t)$ of the operator $H$, after passing to a companion form \eqref{f-companion}, together satisfy a Fuchsian system of linear ordinary differential equation. This follows from the direct inspection of the way the highest order derivatives of $u_k$ enter the expressions in \eqref{conj-e}.

The identity  \eqref{conj-e},
using the commutation relationship in the Weyl algebra
\begin{equation}\label{wc}
\eu f=f\eu+ g,\qquad  g=\eu(f)\in\O(\C,0)\text{ the Euler derivative of $f$},
\end{equation}
can be rewritten as equality of two differential operators $$\sum_{j=0}^{2n-1} l_j\eu^j=\sum_{j=0}^{2n-1}r_j\eu^j$$ of order $2n-1$, implying the identical coincidence of their coefficients, $l_j=r_j$. Thus we have a system of $2n$ linear ordinary differential equations of order $n$ involving $2n$ unknown functions $u_k,v_k$ and their derivatives. We will show that this system can be reduced to a Fuchsian system of $n^2$ differential equations of order $1$.

One can instantly verify that these equations have the following structure.
\begin{enumerate}
 \item All expressions for $l_j$ are linear with respect to the functions $u_k=u_{k,0}$ and their iterated Euler derivatives $u_{ki}=\eu u_{k,i-1}$ of orders $1\le i\le n$ with holomorphic coefficients.
 \item All expressions for $r_j$ are linear with respect to the functions $v_k$ with holomorphic coefficients.
\end{enumerate}
It is rather easy to control the coefficients with which the \emph{highest order} derivatives $u_{kn}$ and $v_k$ enter these equations.

The coefficients with which the variables $v_k$ enter the linear forms $r_j$, form an ``upper triangular'' $n\times 2n$-matrix with the same invertible diagonal entry $a_n$: the highest number forms $r_{2n-1},\dots,r_{2n-k}$ depend only on the variables $v_{n-1},\dots,v_{n-k}$, and the variable $v_{n-k}$ enters with the coefficient $a_n$ for all $k=1,\dots,n$.

The coefficients with which the highest order derivatives $u_{kn}$ enter the linear forms $l_j$, are zero for $l_{2n-1},\dots, l_{n}$ and form a ``diagonal'' $n\times 2n$-matrix with the same invertible diagonal entry $b_n$ in the forms $l_{n-1},\dots,l_0$. Indeed, the formulas \eqref{wc} imply that a highest order derivative $u_{kn}$ can appear only after iterated transposition with the term $b_n\eu^n$ and only before the powers of the type $\eu^{j-n}$.

Together these two observations imply that the system of the linear equations $l_j=r_j$, $j=2n-1,\dots,1,0$ can be resolved with respect to the variables $u_{kn},v_k$, in particular,
\begin{equation}\label{fsc}
u_{kn}(t)=\sum_{i=0}^{n-1}\sum_{j=0}^{n-1}c_{knij}(t)u_{ij}(t), \qquad c_{knij}\in\O(\C,0),\ k=0,\dots,n-1
\end{equation}
(and of course similar expressions for the $v_k$).

This system of $n$ linear ordinary differential equations of order $n$ with respect to the functions $u_k=u_k(t)$ is explicitly resolved with respect to the highest order derivatives, hence is a \emph{Fuchsian} system of $n^2$ first order equations in exactly the same way as in \eqref{f-companion}.

It remains only to refer to the well-known fact: any formal solution of a Fuchsian system converges, see \cite{thebook}*{Lemma~16.17 and Theorem~16.16}. Thus any noncommutative series for an operator $H$ conjugating two Fuchsian operators $L,M$, converges. Convergence of the series for $K$ follows by the uniqueness of the right division of $MH$ by $L$.
\end{proof}

\end{document}